\documentclass[12pt, a4paper, reqno, twoside, makeidx]{amsart}
\usepackage[margin=3cm, marginpar=3cm]{geometry}

\usepackage{verbatim}
\usepackage{amsmath}
\usepackage{amsfonts}
\usepackage{amssymb}
\usepackage{graphicx}
\usepackage{color}
\usepackage{empheq}
\usepackage{braket}
\usepackage[font={small,it}]{caption}
\usepackage{xypic} 
\usepackage[all]{xy}

\usepackage{amssymb}
\newcommand*{\QEDB}{\hfill\ensuremath{\square}}%

\usepackage{tikz}
\usetikzlibrary{calc}

\theoremstyle{remark}

\theoremstyle{plain}
\newtheorem{thm}{Theorem}[section]
\newtheorem{lem}[thm]{Lemma}
\newtheorem{prop}[thm]{Proposition}
\newtheorem{defi}[thm]{Definition}
\newtheorem{cor}[thm]{Corollary}

\theoremstyle{definition}
\newtheorem{exa}[thm]{Example}

\newcommand{\Z}{\mathbb{Z}}

\newcommand{\R}{\mathbb{R}}
\newcommand{\C}{\mathbb{C}}

\newcommand{\Laur}{\mathbb{Z}_2[U,U^{-1}]}

\newcommand{\F}{\mathbb{Z}_2}
\newcommand{\Zc}{\mathcal{Z}}

\newcommand{\x}{\mathbf{x}}
\newcommand{\y}{\mathbf{y}}
\newcommand{\gr}{\text{gr}}

\newcommand{\Spinc}{\text{Spin}^c}

\newcommand{\CFK}{\mathcal{CFK}}

\begin{document}
\title[Upsilon type invariants]{Upsilon type concordance invariants}
\author{Antonio Alfieri}
%\address{Central European University}
%\email{alfieri\_antonio@phd.ceu.edu}
\begin{abstract}To a region $C$ of the plane satisfying a suitable convexity condition we associate a knot concordance invariant $\Upsilon^C$.  For appropriate choices of the domain this construction gives back some known knot Floer concordance invariants like Rasmussen's $h_i$ invariants, and the Ozsv\' ath-Stipsicz-Szab\' o upsilon invariant. Furthermore, to three such regions $C$, $C^+$ and $C^- $ we associate invariants $\Upsilon_{C^\pm, C}$ generalising Kim-Livingston secondary invariant. We show how to compute these invariants for some interesting classes of knots (including alternating and torus knots), and we use them to obstruct concordances to Floer thin knots and algebraic knots. 
\end{abstract}
\maketitle
\thispagestyle{empty}

\section{Introduction}
In \cite{OS2} Ozsv\' ath and Szab\' o, by essentially studying the Floer homology \cite{floer} of certain Lagrangian tori in the $g$-fold symmetric product of a genus $g$ Riemann surface, found a package of three-manifold invariants called Heegaard Floer homology. In \cite{OS7} they used this circle of ideas to define a related package of knot invariants named knot Floer homology. See \cite{grid} for an extensive exposition of this topic. 

Knot Floer homology has been used to produce knot concordance invariants by many authors \cite{Ras1,tau,OSS4,KimLiv}. The purpose of this note is to show that all these constructions can be seen as particular cases of a more general construction. Our investigation is mainly motivated by the following  applications.

\subsection{}
In \cite{lindman} Lidman and Moore characterized $L$-space pretzel knots. They found that a pretzel knot has an $L$-space surgery if and only if it is a torus knot $T_{2, 2n+1}$ for some $n \geq 1$, or a pretzel knot in the form $P(-2, 3, q)$ for some $q \geq 7$ odd. Motivated by the exploration started by Wang \cite{wang2}, and Livingston \cite{Livingston2} one may wonder if $L$-space pretzel knots of the form $P(-2,3,q)$ are concordant to algebraic knots. 

\begin{thm}\label{applicaton}None of the $L$-space pretzel knots  $P(-2,3,q)$, with $q\geq 7$ odd, is conconcordant to a sum of algebraic knots. 
\end{thm}

Notice that for these knots the obstruction found in \cite[Corollary 3.5]{wang2} vanish.

\subsection{} In \cite{FLZalternating} Friedl, Livingston and Zentner asked whenever a sum of torus knots is concordant to an alternating knot. In \cite{zemke} Zemke used involutive Floer homology \cite{manoleacuendriks} to prove that certain connected sums of torus knots are not concordant to Floer thin knots. 
Floer thin knots are upsilon-alternating, meaning that $\Upsilon_K(t) = -\tau(K) \cdot (1-|1-t|)$. A straightforward argument shows that a sum of \textit{positive} torus knots   is upsilon-alternating if and only if it is a connected sum of $(2, 2n+1)$ torus knots and indeed alternating. However, when both positive and negative torus knots are involved this obstruction can vanish.

\begin{prop} \label{applicationthin} The knot $K=T_{8,5}\# - T_{6,5} \# -T_{4,3}$ is upsilon-alternating but not concordant to a Floer thin knot. 
\end{prop}

\noindent
The connected sum formula (Theorem \ref{sum}) employed in the proof of Proposition \ref{applicationthin} is  used in \cite{paoloalt} to decide which sums of two torus are concordant to alternating knots.

\section{A quick review of knot Floer homology}\label{review}
An \textbf{Alexander filtered, Maslov graded chain complex} is a finitely-generated, $\Z$-graded, $(\Z \oplus \Z)$-filtered chain complex $C= (\bigoplus_{\x \in B} \F[U, U^{-1}], \partial )$ such that
\begin{itemize}
\item $\partial$ is $\Laur$-linear and given a basis element $\x \in B$,  $\partial \x = \sum_\y n_{\x, \y}U^{m_{\x,\y}} \cdot \y$ for suitable coefficients $ n_{\x, \y} \in \F$, and non-negative exponents $m_{\x, \y} \geq 0$,
\item the multiplication by $U$ drops the homological (Maslov) grading $M$ by two, and the filtration levels (denoted by $A$ and $j$) by one.
\end{itemize} 
An Alexander filtered, Maslov graded chain complex is said of \textbf{knot type} if in addition $H_*(C, \partial)= \F[U, U^{-1}]$  graded so that $\text{deg}U=-2$.
An Alexander filtered, Maslov graded chain complex can be pictorially described as follows: 
\begin{enumerate}
\item picture each $\F$-generator $U^m \cdot \x$ of $C$ on the planar lattice $\Z \times \Z \subset \R^2$ in position $\left(A(\x)-m, -m \right) \in \Z \times \Z$,
\item label each $\Z_2$-generator $U^m \cdot \x$ of $C$ with its Maslov grading $M(\x)-2m\in \Z$,
\item connect two $\F$-generators $U^n \cdot \x$ and $U^m \cdot \y $ with a directed arrow if in the differential of $U^n \cdot \x$ the coefficient of $U^m \cdot \y$ is non-zero.
\end{enumerate}
In \cite{OS7} Ozsv\' ath and Szab\' o show how to associate to a knot $K \subset S^3$ a knot type complex $CFK^\infty(K)$ whose filtered chain homotopy type only depends on the isotopy class of $K$. For a concise introduction to the background material see \cite{HFKsurvey}.

\subsection{Hom's invariance principle} 
Denote by $\CFK$ the set of knot type complexes up to filtered chain homotopy. Say that two knot type complexes are \textbf{stably equivalent} $C_1 \sim C_2$  if there exist Alexander filtered, Maslov graded, acyclic chain complexes $A_1$ and $A_2$ such that   
$C_1 \oplus A_1 \simeq C_2 \oplus  A_2$.
The quotient set $\CFK /_\sim$ has a natural group structure: the sum is given by tensor product, the class of zero is the one represented by the Floer chain complex of the unknot $CFK^\infty(U)$, and the inverse of the class of a complex $C$ is the one represented by its dual complex $\text{Hom}(C,\F[U, U^{-1}])$. 

\begin{thm}[Hom \cite{HFKsurvey}]\label{jen}The map $K \mapsto CFK^\infty(K)$ associating to a knot $K \subset S^3$ its knot Floer complex descends to a group homomorphism $\mathcal{C} \to \CFK/_\sim$. 
\end{thm}

Summarizing, in order to produce a concordance invariant $\mathcal{C} \to \Z$ one only needs to produce a map $f : \CFK  \to \Z$ such that $f(C_* \oplus A_*)=f(C_*)$ for every Alexander filtered, Maslov graded, acyclic chain complex $A_*$.  

\section{Upsilon type invariants} \label{mainsection}
Inspired by the exposition in \cite{Livingston1} we use knot Floer homology to define some more  concordance invariants. We start with a definition.

\begin{defi} \label{SW} A region of the plane $C \subset \R^2$ is said to be a south-west region if it is non-empty and $(\overline{x}, \overline{y}) \in C \Rightarrow \{ (x,y) \ | \ x\leq \overline{x}, y \leq \overline{y}\} \subseteq C$. 
\end{defi}

Let $C$ be a south-west region of the plane. For $t \in \R $ let $C_t=\{(x,y)\ | \  (x-t, y-t) \in C\}$ denote the translate of $C$ in the direction of $v_t=(t,t)$. 
%Notice that $C_t$ is a south-west region for each $t\in \R$.
Given a knot type complex $K_*$ consider the map induced on $H_0$ by the inclusion $K_*(C_t) \hookrightarrow  K_*$, where $K_*(C_t)$ denotes the subcomplex spanned by the generators of $K_*$ lying in $C_t$. 
%(notice that the condition of Definition \ref{SW} guarantees that $K_*(C_t)$ is a subcomplex).
%Notice that since $C_t$ is a south-west region $K_*(C_t) \subset K_*$ is a subcomplex.
Since $C_t \subseteq C_{t'}$ for $t \leq t'$, and $\bigcup_{t \in \R} C_t= \R^2$, a cycle representing the generator of $H_0(K_*)=\F$ will eventually be contained in $K_*(C_t)$. Thus, for $t$ big enough the inclusion $H_0(K_*(C_t)) \to H_0( K_*)$ is a surjective map. Let $\Upsilon^C(K_*)$ be the minimum $t\in \R$ such that $K_*(C_t) \hookrightarrow  K_*$ induces a surjection on $H_0$. Here we are using the Maslov grading as homological grading so that $H_{2i}(K_*)=\F$ and zero otherwise.

\begin{lem} \label{amain} Suppose that  $C$ is a south-west region. If $K_*$ and $K'_*$ are two stably equivalent knot type complexes then $\Upsilon^C(K_*)=\Upsilon^C(K'_*)$.
\end{lem} 
\begin{proof}
The surjectivity of the map induced in homology by the inclusion $K_*(C_t) \hookrightarrow  K_*$ is not infected if we sum an acyclic complex $A$ on the right and a subcomplex of the \textit{same} acyclic on the left. 
\end{proof}

\begin{cor}\label{invariance}Suppose that  $C \subset \R^2$ is a south-west region. Given a knot $K\subseteq S^3$ set $\Upsilon^C(K)= \Upsilon^C(CFK^\infty(K))$. Then $\Upsilon^C(K)$ is a concordance invariant. \QEDB
\end{cor}

\subsection{The classical upsilon invariant}\label{classical}
Choose the lower half-space
\[H_t = \left\{ \frac{t}{2}\cdot A+\left(1-\frac{t}{2} \right) \cdot j \leq 0  \right\} \]
as south-west region. As $t$ ranges in $[0,2]$ we get a one-parameter family of invariants of knot type complexes $\Upsilon_t(K_*)=\Upsilon^{H_t}(K_*)$. According to Corollary \ref{invariance} this provides a one-parameter family of knot concordance invariants. More specifically, set \[\Upsilon_K(t)=-2 \cdot \Upsilon^{H_t}(CFK^\infty(K)) \ . \] 
In \cite[Section 14]{Livingston1}  Livingston proves that the invariant $\Upsilon_K(t)$ agrees with the upsilon invariant defined by Ozsv\' ath, Stipsicz and Szab\' o \cite{OSS4}. 

\subsection{Regions for Rasmussen's $h_i$ invariants}\label{Vi}
For $s\geq 0$ choose as south-west region $Q_s=\{A\leq s , j\leq 0\}$. This leads to a one-parameter family of knot concordance invariants $V_K(s)=-2\Upsilon^{Q_s}(K)$. These concordance invariants are equivalent to the one introduced by Rasmussen in \cite{Ras1}. We justify the equivalence by proving that the same relation with correction terms pointed out and discussed in \cite{NiWu} holds. 
  
\begin{prop}Let $K \subseteq S^3$ be a knot and $q\geq 2g(K) -1$ be an integer. Denote by $W_q(K)$ the $q$-framed two-handle attachment along $K$ to $D^4$, so that $S^3_q(K)= \partial W_q(K)$. For any integer $m \in [-q/2, q/2)$ let $\mathfrak{s}_m\in \Spinc(S^3_q(K))$ denote the restriction to $S^3_q(K)$ of a $\Spinc$ structure $\mathfrak{t}_m$ on $W_q(K)$ such that
$ \braket{ c_1(\mathfrak{s}), [\widehat{F} ] }+q=2m$,
where $\widehat{F} \subset W_q(K) $ denotes a capped-off Seifert surface for $K$. Then
\[ d(S^3_q(K), \mathfrak{s}_m)= \frac{(q-2m)^2-q}{4q}+ V_K(m) \ ,\]
where $d$ denotes the Heegaard Floer correction term introduced in \cite{OS24}. 
\end{prop} 
\begin{proof} Suppose that $z_1, \dots , z_k \in CFK^\infty(K)$ are the cycles with Maslov grading zero representing the generator of $H_0(CFK^\infty(K))=\F$. If $Q_{s,t}$ denotes the translate of $Q_s=\{A\leq s , j\leq 0\}$ in the $(t,t)$-direction then  $Q_{s,t} =\{\max(A-s,j)\leq t\}$. Thus, 
\begin{equation}\label{formulahi}
V_K(s)= -2 \cdot \min_{i} \max(A(z_i)-s, j(z_i)) 
\end{equation} 
We now prove that the very same min-max formula can be used to compute the correction terms of the $q$-framed surgery. 

Let $q$ and $m$ be as above. Remove a ball from $W_q(K)$, turn the resulting two-handle cobordism upside-down, and change orientation in order to get a cobordism $X: S^3_q(K) \to S^3$. According to \cite[Theorem 4.4]{OS7}, the map induced in homology by the inclusion $C\{\max(A-m,j)\leq 0\} \hookrightarrow C\{j \leq 0\}$ represents the map $F_X: HF^-(S^3_q(K), \mathfrak{s}_m) \to HF^-(S^3)$ induced by $X$. Since 
\[\gr(F_X(\xi) ) - \gr(\xi)= \frac{c_1(\mathfrak{s}_m)^2- 2 \chi(X) - 3 \sigma(X)}{4}  \ ,\]
where $\xi$ denotes the generator of the tower of $HF^-(S^3_q(K), \mathfrak{s}_m)$, we can conlude that 
\[d(S^3_q(K), \mathfrak{s}_m)= d+\frac{(q-2m)^2-q}{4q}   \ , \]
where $d$ denotes the Maslov grading of the generator of the tower of 
$H_*(C\{\max(A-m,j)\leq 0\}) \simeq HF^-(S^3_q(K), \mathfrak{s}_m)$.

Since the inclusion  $C\{\max(A-m,j)\leq 0\} \hookrightarrow C\{j \leq 0\}$ sends the generator of the tower of $HF^-(S^3_q(K), \mathfrak{s}_m)$ to a $U^n$-multiple of the one of $HF^-(S^3)\simeq H_*(C\{j \leq 0\})$, if $z_1, \dots , z_k \in CFK^\infty(K)$ denote the Maslov grading zero cycles for the generator of $H_*(CFK^\infty(K))$ we have that $d= \max_i M(U^{n_i}\cdot z_i)= \max_i M(U^{n_i} \cdot z_k)- 2n_i= -2 \min_i n_i$, 
where $n_i$ is the minimum $n\geq 0$ such that $U^n \cdot z_i \in C\{\max(A-m,j)\leq 0\}$. 
Since $n_i=\max(A(z_i)-m,j(z_i))$ this proves that $d=-2\min_i  \max(A(z_i)-m ,j(z_i))=V_K(m)$, and we are done. 
\end{proof}

\subsection{Estimates on the slice genus}
Suppose that $C$ is a south-west region.  Associated to $C$ there is a height function 	\[h_C(x)= \min\{t\in \R \text{ such that } (x, 0) \in C_t \} \ , \] 
where $C_t$ denotes as usual the translate of $C$ in the $v_t= (t,t)$ direction. The height function $h_C$ relates the upsilon invariant of the region $C$ to the slice genus. 

\begin{thm}Let $C$ be a south-west region. Given a knot $K \subset S^3$ the inequality 
\begin{equation}\label{extimate}
\max\{ \Upsilon^C(K), \Upsilon^C(-K)\} \leq h_C(g_4(K))
\end{equation} 
holds,  where $g_4(K)$ denotes the slice genus of $K$.
\end{thm}
\begin{proof}
First of all notice that $h_C(x)$ is a monotone increasing function: since $C$ is a south-west region, $(x-\delta , 0 ) \in C_{\Upsilon^C(K)}$ for $\delta>0 $. Thus, $h_C(x-\delta )\leq h_C(x)$.

Let $\nu^+= \min_i\{V_K(i)=0 \}$. From the definition of the height function $h_C(x)$ and the fact that $C$ is a south-west region one immediately conclude that $\{A \leq \nu^+ ,  j \leq 0\}  \subseteq C_{h_C(\nu^+)}$. The fact that $V_K(\nu^+)=0$ ensures that the south-west region $\{A \leq \nu^+ ,  j \leq 0\} $ contains a cycle generating $H_0(CFK^\infty(K))$ and consequently (because of the inclusion) that so does the translate $C_{h_C(\nu^+)}$. This proves that $\Upsilon^C(K) \leq h_C(\nu^+)$. On the other hand, according to Rasmussen \cite[Corollary 7.4]{Ras1} $\nu^+ \leq g_4(K)$, thus $\Upsilon^C(K) \leq h_C(\nu^+) \leq h_C(g_4(K))$. 

By doing the same argument for $-K$ instead of $K$ we get that $\Upsilon^C(-K) \leq  h_C(g_4(-K))=h_C(g_4(K))$, and we are done.
\end{proof}

\begin{exa}
If we choose $C=\{t/2A+(1-t/2)j\leq 0\}$ as in the classical upsilon invariant (Section \ref{classical}) one has $h_C(x)=t/2 \cdot x$. In this case Equation \ref{extimate} leads to the inequality $|\Upsilon_K(t)|=2 \cdot \max\{\Upsilon^C(K),\Upsilon_K^C(-K)\}\leq 2 h_C(g_4(K)) =t g_4(K)$,
where the first identity is due to the identity  $\Upsilon_K^C(-K)=-\Upsilon^C(K)$ (which is not valid for any $C$). Compare this with \cite[Theorem 1.11]{OSS4}.
\end{exa}
 
\section{Secondary invariants}\label{secondaryA}
Roughly speaking, upsilon type invariants measure how far one needs to travel north-east in the $(A,j)$ plane in order to see a cycle generating $H_0(CFK^\infty)$ appear. As suggested by Kim and Livingston in \cite{KimLiv}, other concordance invariants could be obtained my measuring how far one should go in order to see realized some expected homologies. 

Suppose that two south-west regions $C^+$ and $C^-$ are given. Given a knot type complex $K_*$  one can consider the maps induced in homology by the inclusions $K_*(C^+_t) \hookrightarrow K_*$ and $K_*(C^+_t) \hookrightarrow K_*$ (here we are using again the notation of the beginning of Section \ref{mainsection}). For $\gamma_\pm=\Upsilon^{C_\pm}(K_*)$ one gets surjections $H_0(K_*(C^+_{\gamma_+})) \to H_0(K_*)$ and $H_0(K_*(C^-_{\gamma_-})) \to H_0(K_*)$. Denote by $\Zc^+$ and $\Zc^-$ the set of cycles in $K_*(C^+_{\gamma_+})$ and $K_*(C^-_{\gamma_-})$ respectively projecting on the generator of $H_0(K_*)$.

Suppose now that a third south-west region $C$ has been fixed. Since $H_0(K_*)=\F$, for $t \in \R$ large enough there will be a 1-chain $\beta \in K_1$ realizing a homology between a 0-cycle in $\Zc^+$ and one in $\Zc^-$. We define $\Upsilon_{C^\pm, C}(K_*)$ as the minimum $t \in \R$ for which a cycle in $\Zc^+$ represents inside 
$K_*(C^+_{\gamma_+})+ K_*(C^-_{\gamma_-}) + K_*(C_t)$ the same homology class of a cycle in $\Zc^-$.  We set $\Upsilon_{C^\pm, C}(K_*)=-\infty$ in the eventuality that $\Zc^+ \cap \Zc^- \not=  \emptyset$.

\begin{lem} \label{invariancesecondary}
Suppose that  $C^+$, $C^-$ and $C$ are given south-west regions. If $K_*$ and $K'_*$ are two stably equivalent knot type complexes then $\Upsilon_{C^\pm, C}(K_*)=\Upsilon_{C^\pm, C}(K'_*)$.
\end{lem} 
\begin{proof} Suppose that $K'_*=K_*\oplus A$ is obtained from $K_*$ by adding an acyclic complex $A$. Set $\gamma_\pm=\Upsilon^{C_\pm}(K_*)=\Upsilon^{C_\pm}(K'_*)$, and denote by $\Zc^\pm(K_*)$ and $\Zc^\pm(K'_*)$ the set of cycles projecting to the generator through $H_0(K_*(C^\pm_{\gamma_\pm})) \to H_0(K_*)$ and $H_0(K'_*(C^\pm_{\gamma_\pm})) \to H_0(K'_*)$ respectively. 

We prove that $\Upsilon_{C^\pm, C}(K_*)=\Upsilon_{C^\pm, C}(K'_*)$ by proving the two inequalities. Suppose by contradiction that there exists $t< \Upsilon_{C^\pm,C}(K^*)$ for which a cycle $z^+ \in\Zc^+(K'_*)$ gets identified with a cycle in $z^-\in \Zc^-(K'_*)$ in $K'_*(C_t) +K'_*(C^+_{\gamma_+})+ K'_*(C^-_{\gamma_-})$. 

Pick a 1-chain $\beta'_t\in K'_*(C_t)+K'_*(C^+_{\gamma_+})+ K'_*(C^-_{\gamma_-})$ such that $z^+- z^-=\partial \beta'_t$, and write $\beta'_t= \beta_t+ a$ with $\beta_t \in K_*(C_t)+K_*(C^+_{\gamma_+})+ K_*(C^-_{\gamma_-})$ and $a \in A$. Notice that  $z^+=z^+_K+a^+$ and $z^-=z^-_K+a^- $, for some $a^+, a^- \in A$, $z^+_K\in \Zc^+(K_*)$,  and $z^-_K \in \Zc^-(K_*)$. 
By rewriting the relation $z^+- z^-=\partial \beta'_t$ we get that $ (z^+_K -z_K^- -\partial \beta_t) + (a^+-a^- - \partial a)=0$,
from where we can conclude that $z^+_K -z_K^-=\partial \beta_t$. This contradicts the fact that $\Upsilon_{C^\pm,C}(K_*)$ is the minimum $t$ for which such an homology exists, and proves that $\Upsilon_{C^\pm,C}(K_*)\leq \Upsilon_{C^\pm,C}(K'_*)$. The reverse inequality has a similar proof.
\end{proof}

\begin{cor}\label{invariance2} For a knot $K \subset S^3$ set $\Upsilon_{C^\pm, C}(K)=\Upsilon_{C^\pm, C}(CFK^\infty(K))$. Then $\Upsilon_{C^\pm, C}(K)$ defines a knot concordance invariant.  \QEDB
\end{cor}

\subsection{Breaking points}\label{secondaryB}
Summarizing, given south-west regions $C^+$, $C^-$ and $C \subset \R^2$ we get a map $\Upsilon_{C^\pm, C}: \CFK /_\sim \to [ -\infty, + \infty)$. In \cite{KimLiv} Kim and Livingston produce south-west regions for which the condition $\Zc^+ \cap \Zc^- =  \emptyset$ is guaranteed.  

\begin{lem}[Kim-Livingston]\label{KL} For $t \in [0,2]$ let $\Upsilon_t: \CFK /_\sim \to \R $ denotes the stable equivalence invariant associated to the lower half-space $H_t$ of Section \ref{classical}. Suppose that $K_*$ is a knot type complex such that $\Upsilon_t(K_*)$ as function of $t \in [0,2]$ is non smooth at $t=t^*$. Furthermore, suppose that the derivative of $\Upsilon_t(K_*)$ at $t=t^*$ has a positive jump, meaning that  
\[ \Delta \Upsilon'_{t}(K_*)= \lim_{\epsilon \to 0} \left( \Upsilon'_{t+\epsilon}(K_*) - \Upsilon'_{t-\epsilon}(K_*) \right)\]
is positive at $t=t^*$. Then for $\delta>0$ small enough $C^-= H_{t^*-\delta}$ and $\ C^+= H_{t^*+\delta}$ give two south-west regions such that $\Zc^+ \cap \Zc^- =  \emptyset$.  \QEDB
\end{lem}  

We say that the upsilon function $\Upsilon_t(K_*)$ of a knot type complex $K_*$ has a \textbf{breaking point} at $t=t^*$ if for a small perturbation $ \delta >0$, $C^-= H_{t^*-\delta}$ and $C^+= H_{t^*+\delta}$ are two south-west regions such that $\Zc^+ \cap \Zc^- =  \emptyset$. In what follows, the cycles in $\Zc^+$ and $\Zc^-$ are referred to as at the positive and the negative \textbf{exceptional cycles} of the breaking point.  Lemma \ref{KL} says that the \textbf{singularities} of $\Upsilon_t(K_*)$ (points where $\Upsilon_t(K_*)$ is non-smooth) at which $\Delta \Upsilon'_{t}(K_*)>0$ are in fact breaking points. 
%Notice that there can be singularities that are not breaking point, and smooth points that are breaking points.

In the notation of Proposition \ref{KL} set 
\begin{equation}\label{KLdef}
\Upsilon_{C,t}^{(2)}(K_*) = -2 \cdot ( \Upsilon_{H_{t \pm \delta}, C}(K_*) - \Upsilon_t(K_*)) \ .
\end{equation}
for $\delta>0$ small enough. This provides a one-parameter family of knot concordance invariants $ \Upsilon_{C,t}^{(2)}(K)= \Upsilon_{C,t}^{(2)}(CFK^\infty(K))$. Notice that the invariant $\Upsilon_{K,t}^{(2)}(s)=\Upsilon_{H_s,t}^{(2)}(K)$ is exactly the secondary upsilon invariant introduced by Kim and Livingston in \cite{KimLiv}.  

\vspace{-0.3cm}
\section{Floer thin knots}\label{thin}
A knot $K \subset S^3$ is called \textbf{Floer thin} if its knot Floer homology groups $\widehat{HFK}_{*,*}(K)$ are concentrated on a diagonal, meaning that $\widehat{HFK}_{i,j}(K)=0$ if $i-j\not= \delta$ for a suitable constant $\delta$. Examples of Floer thin knots are alternating and quasi-alternating knots \cite{OS8,OzsMan} (in these cases $\delta=-\sigma/2$, where $\sigma$ denotes the knot signature). 
In \cite{Petkova} Petkova shows that for a Floer thin knot the chain homotopy type of  $CFK^\infty(K)$ can be completely reconstructed from its Ozsv\' ath-Szab\' o tau invariant $\tau=\tau(K)$ and its Alexander polynomial $\Delta_K=a_0+ \sum_{s>0} a_s (T^s+T^{-s})$. More precisely we have that:
\begin{itemize}
\item $CFK^\infty(K)$ has exactly $|a_s|$ generators with $A=s$ and $j=0$,
\item $CFK^\infty(K)=(S_\tau \otimes \F[U,U^{-1}]) \oplus (\bigoplus_i Q_i \otimes \F[U,U^{-1}])$, where $S_\tau$ is a staircase complex, and the $Q_i$'s are square complexes as the one shown in Figure \ref{signs}.
\end{itemize}

Notice that $A=\bigoplus_i Q_i \otimes \F[U,U^{-1}]$ is acyclic. Consequently, up to acyclics, for a Floer thin knot $K$ we have that $CFK^\infty(K)= S_{\tau(K)} \otimes \F[U,U^{-1}]$. 

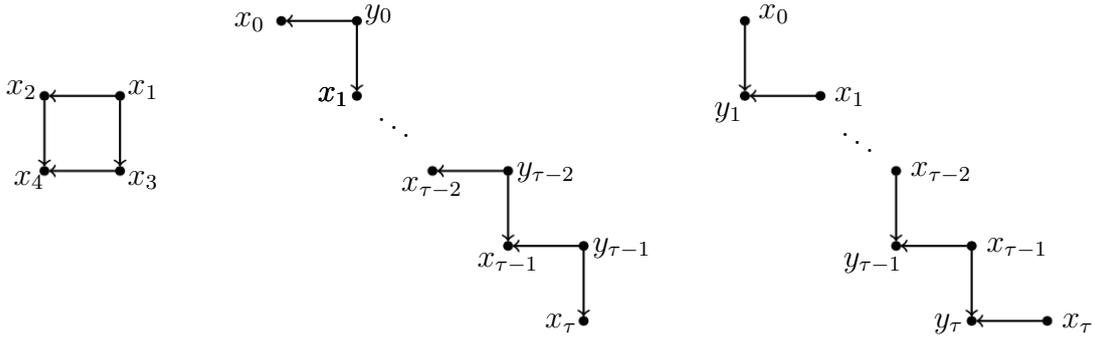
\begin{figure}[t]
\center 
\begin{tikzpicture}
       \node at (0,0) {\includegraphics[width=0.9\textwidth]{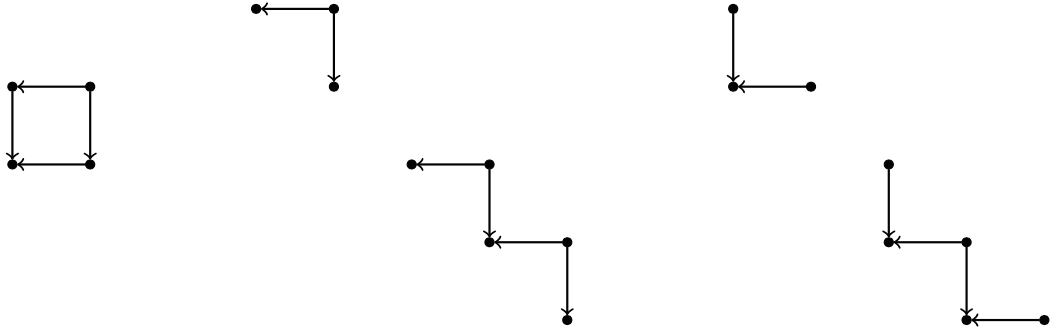}};
       \node at (3,2.1) {$x_0$};
       \node at (2.4,0.8) {$y_1$};
       \node at (4,1) {$x_1$};
       \node at (4.1,0.5) {$\ddots$};
       
       \node at (5.2,0) {$x_{\tau-2}$};
       \node at (6.2,-1) {$x_{\tau-1}$};
      \node at (4.3,-1.2) {$y_{\tau-1}$};
       \node at (7,-2) {$x_\tau$};
        \node at (5.3,-2) {$y_{\tau}$};
    
       \node at (-2.8,1) {$x_1$}; 
        \node at (-3.9,2) {$x_0$};
       \node at (-2.2,2.1) {$y_0$};
       \node at (-2.8,1) {$x_1$};
       \node at (-2,0.7) {$\ddots$};
       \node at (-2.8,1) {$x_1$}; 
       
       \node at (0,0) {$y_{\tau-2}$};
       \node at (1,-1) {$y_{\tau-1}$};
      \node at (-1.5,-0.2) {$x_{\tau-2}$};
       \node at (0.2,-2) {$x_\tau$};
        \node at (-0.5,-1.2) {$x_{\tau-1}$};
        
       \node at (-6.9,1.1) {$x_2$}; 
        \node at (-5.3,1.1) {$x_1$};
       \node at (-6.8,-0.1)  {$x_4$};
       \node at (-5.3,-0.1) {$x_3$}; 
\end{tikzpicture}   
\caption{\label{signs} The square complex $Q$ (left), and the staircase complex $S_\tau$  ($\tau \leq 0$ on the center, $\tau>0$ on the right).}
\end{figure}

\subsection{Three-parameter upsilon invariants of thin knots} 
We show how to compute some upsilon type invariants in the case of Floer thin knots.
Choose as south-west region 
\[ C=\left\{\frac{t}{2}\cdot A+\left(1-\frac{t}{2} \right) \cdot j\leq 0 \right\}   \cup \left\{ \frac{s}{2}\cdot A+\left(1-\frac{s}{2} \right) \cdot j \leq\phantom{\frac{t}{2}} \hspace{-0.3cm} q\right\} \ .\] 
As the parameters $s,t \in [0,1]$ and $q \in \R$ vary, the concordance invariant $\Upsilon^C$ gives rise to a three-parameter family of concordance invariants $\Upsilon_K(t,s,q)$ collapsing to the classical upsilon invariant when $t=s$ and $q=0$. Let us compute $\Upsilon^C(K)= \Upsilon^C(CFK^\infty(K))$ for a Floer thin knot $K$.

Suppose first that $\tau= \tau(K)$ is positive.  
Since $CFK^\infty(K)= S_\tau \otimes \F[U,U^{-1}] \oplus A$ with $S_\tau$ staircase shaped as in Figure \ref{signs}  and $A$ acyclic, Hom's principle shows that $\Upsilon^C(CFK^\infty(K))=\Upsilon^C(S_\tau \otimes \Laur)$. 

Let $C_\gamma$ denote the translate of $C$ in the $v_\gamma=(\gamma, \gamma)$ direction. The south-west region $C_\gamma$ contains a generator of $H_0(S_\tau \otimes \Laur) = \F$ as soon as it contains one of the $x_i$ generators of Figure \ref{signs}. Thus, $\Upsilon^C(S_\tau \otimes \Laur)$ is the minimum $\gamma$ such that 
\[\frac{t}{2} \cdot A(x_i)+\left(1-\frac{t}{2} \right) \cdot j(x_i)\leq \gamma    \ \ \text{ or }  \ \ \frac{s}{2}\cdot A(x_i)+\left(1-\frac{s}{2} \right) \cdot j(x_i)-q \leq \gamma \] 
for at least one of the $x_i$ generators, hence $\Upsilon^C(S_\tau \otimes \Laur)$ is computed by the expression:
\[  \min_i \min \left\{\frac{t}{2}  A(x_i)+\left(1-\frac{t}{2} \right)  j(x_i), \frac{s}{2} A(x_i)+\left(1-\frac{s}{2} \right)  j(x_i)-q\right\} \ .\]
Plugging in $A(x_{i})=\tau-i$ and $j(x_{i})=i$ we get 
$\Upsilon^C(S_\tau \otimes \Laur)= \min_i \min \{ (1-t)i + \tau, (1-s)i + \tau-q \}$ from where the identity 
\[ \Upsilon^C(S_\tau \otimes \Laur)= \min \left\{ 
\frac{t}{2} \tau, 
(1-s) \left\lceil \frac{\tau}{2}+\frac{q}{t-s}\right\rceil +\frac{s\tau -2q}{2} \right\}\]
follows for $t\not=s$. For $t=s$  one can easily see that
\[\Upsilon^C(S_\tau \otimes \Laur)= 
\frac{t}{2} \tau - \max \{0, q \}.\]

If $\tau<0$ then the situation is somehow easier: there is only one 0-cycle generating $H_0(S_\tau \otimes \Laur)= \F$, namely $z=\sum_i x_i$. Thus, in this case $\Upsilon^C(S_\tau \otimes \Laur)$ is computed by the following expression:
\[  \max_i \min \left\{\frac{t}{2}  A(x_i)+\left(1-\frac{t}{2} \right)  j(x_i), \frac{s}{2} A(x_i)+\left(1-\frac{s}{2} \right)  j(x_i)-q\right\} \ .\]
By substituting the values of $A(x_{i})$ and $j(x_{i})$ we get 
\[ \Upsilon^C(S_\tau \otimes \Laur)= \min \left\{ 
(1-s) \left\lfloor \frac{\tau}{2}+\frac{q}{t-s}\right\rfloor +\frac{s\tau -2q}{2}, \frac{2-t}{2} \tau \right\}\]  
if $t\not=s$. If $t=s$  we have the identity 
\[\Upsilon^C(S_\tau \otimes \Laur)= 
\frac{2-t}{2} \tau - \min \{0, q \}.\]
As an immediate corollary of this discussion we get the following proposition.
%(compare with \cite[Theorem 1.14]{OSS4}).
\begin{prop}\label{alpha}Suppose that $K \subset S^3$ is a Floer thin knot. Then,
\[ \Upsilon_K(t) = -\tau(K) \cdot (1-|t-1|) \ .   \] 

\vspace{-0.6cm}
\QEDB 
\end{prop}

Notice that the Ozsv\' ath-Stipsicz-Szab\' o  upsilon function of a Floer thin knot $K \subset S^3$ has only one singularity at $t=1$ where it actually has a breaking point if $\Delta\Upsilon'_{t=1}(K)=2\tau(K)>0$. We now compute the Kim-Livingston secondary invariant of these singularities.

\begin{prop}Suppose that $K \subset S^3$ is a Floer thin knot. Then
\[ \Upsilon^{(2)}_{K,1}(s) =(1-\tau(K))\cdot |1-s| -1 \]
if $\tau(K)>0$, and $\Upsilon^{(2)}_{K,1}(s) =-\infty$ otherwise.
\end{prop}
\begin{proof}In the notation of Section \ref{secondaryB}, we would like to compute
$\Upsilon_{C,1}^{(2)}(K_*)$ for $C=\{s/2A+ (1-s/2)j \leq 0\}$ and $K_*=S_\tau \otimes \Laur$. 

If $\tau>0$, the upsilon function $\Upsilon_t(S_\tau \otimes \Laur)$ has a breaking point at $t=1$.  The exceptional sets $\Zc^+$ and $\Zc^-$ of this breaking point are easy to identify: there is only one positive and one negative exceptional cycle, namely  $z^+=x_0$ and $z^-= x_{\tau}$ (see Figure \ref{thin} again). A quick inspection of the same Figure reveals that a 1-chain realising a homology between $z^+$ and $z^-$ is given by $b=\sum_i y_i$. Notice that there is exactly one such chain since $H_1(S_\tau \otimes \Laur)=0$, and $\partial$ vanishes  on chains with even Maslov grading. Thus,
\[ \Upsilon_{C,1}^{(2)}(S_\tau \otimes \Laur)= -2\left( \max_i \left( \frac{s}{2} A(y_i) +\left(1-\frac{s}{2} \right) j(y_i) \right) - \frac{\tau}{2} \right) . \] 
Plugging in $A(y_i)=\tau-i$ and $j(y_i)=i+1$ the claimed identity
can be deduced by algebraic manipulation. 

If $\tau\leq 0$, there is only one 0-cycle  generating $H_0(S_\tau \otimes \Laur)$, namely $z= \sum_i x_i$. Thus $\Zc^+ \cap \Zc^-= \{ z\}$ and we conclude that $\Upsilon^{(2)}_{K,1}(s)=-\infty$.
\end{proof}

\section{$L$-space knots}\label{L-spaceknots}
Another interesting class of knots is provided by $L$-space knots. Recall that a rational homology sphere $Y$ is an \textbf{$L$-space} if $\widehat{HF}(Y, \mathfrak{s})= \F$ in every $\Spinc$ structure. This happens for example in the case of a lens space $Y=L(p,q)$ whence the name.  A knot $K \subset S^3$ is said to be an \textbf{$L$-space knot} if it has a positive surgery $S^3_p(K)$ that is an $L$-space. Basic examples of $L$-space knots are positive torus knots.

The homotopy type of the master complex of an $L$-space knot can be reconstructed from its Alexander polynomial. More precisely suppose that $K \subset S^3$ is a genus $g$ $L$-space knot. According to \cite{OS3} its Alexander polynomial can be written in the form $\Delta_K(t)  = 1-t^{\alpha_1}+ \dots -t^{\alpha_{2k-1}}+t^{\alpha_{2k}}$,
with $0=\alpha_0<\alpha_1< \dots <\alpha_{2k}=2g$. 
Staring from the sequence $a_i=\alpha_{i}-\alpha_{i-1}$
recording the jumps between consecutive exponents of the monomials appearing in the Alexander polynomial, construct a chain complex $S_*(K)= S_*(a_1, \dots, a_{2k} )$ as follows. Set $S_*(a_1, \dots, a_{2k} )= \F\{x_0, \dots , x_k, y_0, \dots , y_{k-1} \} \otimes \Laur$, and consider the differential  
\[
\begin{cases}
\ \partial x_i= 0 \ \ \ i=0, \dots , k   \\  
\ \partial y_i=x_i+ x_{i+1} \ \ \ i=0, \dots ,  k-1 
 \end{cases} \  .
\]
Define 
\begin{equation}
\begin{cases}
 \  A(x_i)=n_i \\ 
 \ j(x_i)= m_i\\
 \ M(x_i)=0 
 \end{cases} \ \ \ 
\   \text{ and } \ \ \  \ \ \ \
\begin{cases}
\ A(y_i)=n_i \\ 
\ j(y_i)=m_{i+1} \\
\ M(y_i)=1 
\end{cases}
\end{equation}
where 
\[ \ 
\begin{cases}
 \ n_i=g-\sum_{j=0}^{i}a_{2j} \\ 
 \ n_0=0
 \end{cases}
\  \ \ \  \ \
\begin{cases}
 \ m_i=\sum_{j=1}^i a_{2j-1} \\ 
 \ m_0=0
 \end{cases} \ , 
\] 
and coherently extend these gradings to $\F\{x_0, \dots , x_k, y_0, \dots , y_{k-1} \} \otimes \Laur$ so that multiplication by $U$ drops the Maslov grading $M$ by $-2$, and the Alexander filtration $A$ as well as the algebraic filtration $j$ by $-1$. In \cite{Peters} Peters proves that there is a chain homotopy equivalence $CFK^\infty (K) \simeq S_*(K)$. 

\subsection{Kim-Livingston secondary invariant of $L$-space knots}
Let us compute the upsilon invariant $\Upsilon_t(S_*)$ of a  staircase complex 
$S_*=  S_*(a_1, \dots , a_{2k})$. Since the lower half-space
$t/2\cdot A+\left(1-t/2 \right) \cdot j\leq \gamma $ contains a cycle generating $H_0(S_*)$ as soon as it contains one of the $x_i$ generators, we have that
\begin{equation} \label{formulaL}
\Upsilon_t(S_*)= \min_i \left\{ \frac{t}{2}n_i+\left(1-\frac{t}{2} \right) m_i\right\} \ .
\end{equation}
Thus, for  an $L$-space knot $K \subset S^3$ one has
\[ 
\Upsilon_K(t)= -2 \cdot  \min_i \left\{ \left(n_i-m_i \right)\frac{t}{2} +m_i\right\}=- \min_i \{ \alpha_i t + m_i \} \, 
\]
as already pointed out by Ozsv\' ath, Stipsicz and Szab\' o in \cite{OSS4}. 

From Equation \ref{formulaL} it is clear where the upsilon function $\Upsilon_t(S_*)$ of a staircase complex $S_*$ has its breaking points. A parameter $t$ is a singularity for $\Upsilon_t(S_*)$ if and only if the $\min$ on the left hand side of (\ref{formulaL}) is realised by more than one index $i$. These singularities are breaking points since at these points $\Delta \Upsilon'_t>0$. Notice that there are no other breaking points since at a regular parameter $t$ the half space     
$t/2\cdot A+\left(1-t/2 \right) \cdot j \leq  \Upsilon_t(S_*)$
contains (on its boundary line) exactly one $x_i$ generator.

\begin{prop}\label{secondaryupsilonstaircases}
Let $S_*=  S_*(a_1, \dots , a_{2k})$ be a staircase complex. Suppose that $t$ is a breaking point of $\Upsilon_t(S_*)$, then 
\[ \Upsilon_{S_*, t}^{(2)}(s)= -2\left( \max_{i_-\leq j < i_-}  \left\{ \frac{s}{2}n_j+\left(1-\frac{s}{2} \right) m_{j+1}\right\} - \Upsilon_t(S_*) \right) \ , \]
where $i_-$ and $i_+$ denote respectively the minimum and the maximum index  realizing the minimum in Equation \ref{formulaL}.
\end{prop}  
\begin{proof} 
If $t$ is a breaking point of $\Upsilon_t(S_*)$ then the half-space 
\[ \frac{t}{2}\cdot A+\left(1-\frac{t}{2} \right) \cdot j \leq \Upsilon_t(S_*)  \]
contains (actually on its boundary line) exactly those $x_i$ generators which have index $i \in \{ 0, \dots , k\}$ realizing the minimum in the expression of Equation \ref{formulaL}. 

The exceptional sets $\Zc^+$ and $\Zc^-$ of such a singularity both contain exactly one 0-cycle: $z^+=x_{i_+}$ and $z_-=x_{i_-}$ respectively. Notice that since $H_1(S_*)=0$, there is only one 1-chain realizing a homology between these cycles, namely $\beta=\sum_{j=i_-}^{i_+-1} y_j$. Thus, in the notation of Equation (\ref{KLdef}) of Section \ref{secondaryB}, we have that 
\[\Upsilon_{H_{t \pm \delta}, H_s}(S_*) = \max_{i_-\leq j < i_+}  \left\{ \frac{s}{2}n_j+\left(1-\frac{s}{2} \right) m_{j+1}\right\} \]
from where the formula follows. 
\end{proof}
 
\subsection{A connected sum formula}
One of the fundamental properties of the Ozsv\' ath-Stipsicz-Szab\' o upsilon invariant is its additivity property 
\[ \Upsilon_t(A_* \otimes B_*) = \Upsilon_t(A_*) + \Upsilon_t (B_*) \ ,\]
turning $\Upsilon_t$ into a group homomorphism from $\mathcal{CFK} /_\sim$ to the group of piecewise linear functions $[0,2] \to \R$. General upsilon type invariants and their secondary counterparts do not enjoy this property. In this section we prove a connected sum formula for the Kim-Livingston secondary invariant of staircase complexes.

\begin{thm}\label{sum}
Let $A_*=S_*(a_1, \dots, a_{2n})$ and $B_*=S_*(b_1, \dots , b_{2m})$ be staircase complexes. Suppose that $\Upsilon_t(A_*)$ has a breaking point at a point $t=s$  where $\Upsilon_t(B_*)$ is smooth. Then,
\[ \Upsilon^{(2)}_{A_* \otimes B_*,s}(s)= \Upsilon^{(2)}_{A_*,s}(s) \ .\]
\end{thm}
\begin{proof}
Denote by $x_0, \dots , x_{n}$ and $z_0, \dots , z_m$ the Maslov grading zero generators of the staircases of $A_*$ and $B_*$ respectively. Similarly, denote by $y_0, \dots , y_{n-1}$ and $w_0, \dots , w_{m-1}$ their Maslov grading one generators. 

The fact that $\Upsilon_t( B_*)$ is smooth at $t=s$ guarantees that the half-space  
\[ \frac{s}{2}\cdot A+\left(1-\frac{s}{2} \right) \cdot j \leq \Upsilon_s(B_*)  \]
only contains (actually on its boundary line) one 0-cycle $z=z_r$ generating $H_0(B_*)$.
Since $A_*$ is a staircase complex, the set of its exceptional cycles $\Zc^+$ and $\Zc^-$ at $t=s$ both include exactly one 0-cycle. Denote those cycles by $x^+=x_k$ and $x^-=x_h$ respectively. In this notation, the set of exceptional cycles of $A_* \otimes B_*$ at $t=s$ are given by $\Zc^+=\{ x_k \otimes z_r\}$ and $\Zc^-=\{ x_h \otimes z_r\}$. 

Given a chain $\xi= \sum_i \xi_i$ with $\xi_1, \dots , \xi_n$  homogenous with respect to both the Alexander and the algebraic grading, set
\[ E_s(\xi)=\max_i \left\{\frac{s}{2}\cdot A(\xi_i)+\left(1-\frac{s}{2} \right) \cdot j(\xi_i) \right\}  \ .\]
In this notation $E_s(x) \leq \gamma$ if and only if the chain $\xi$ is contained in the subcomplex of the lower half-space $s/2\cdot A+\left(1-s/2 \right) \leq \gamma$. 

It is easy to find a 1-chain realizing a homology between $x_k \otimes z_r$ and $ x_h \otimes z_r$:
\[\partial \left( \sum_{\ell=a}^{b-1} y_\ell \otimes z_r\right)=\sum_{\ell=a}^{b-1} \partial y_\ell \otimes z_r =\sum_{\ell=a}^{b-1} (x_\ell+x_{\ell+1} )\otimes z_r=x_k\otimes z_r -  x_h\otimes z_r  \ . \] 
If we prove that between the 1-cycles realizing a homology between $x_k \otimes z_r$ and $ x_h \otimes z_r$ this is the one with minimal $E_s$ then we conclude that 
\begin{align*}
\Upsilon^{(2)}_{A_* \otimes B_*,s}(s) &=-2 \left( E_s \left(\sum_{\ell=a}^{b-1} y_\ell \otimes z_r \right) - \Upsilon_t(A_* \otimes B_*) \right) \\
&= -2 \left( E_s \left(\sum_{\ell=a}^{b-1} y_\ell\right) +E_s( z_r) - \Upsilon_t(A_*) - \Upsilon_t (B_*) \right)\\
&=-2 \left( E_s \left(\sum_{\ell=a}^{b-1} y_\ell\right) - \Upsilon_t(A_*)  \right) =\Upsilon^{(2)}_{A_*,s}(s)
\end{align*}
and we are done. Let us prove that $\beta=\sum_{\ell=a}^{b-1} y_\ell \otimes z_r$ is a cycle minimizing $E_s(\beta)$ in the class of 1-cycles realizing homologies between $x_k \otimes z_r$ and $ x_h \otimes z_r$.

From the fact that for a staircase complex $\partial(x)=0$ for those $x$'s with homogenous  even Maslov degree, one can conclude that any 1-chain realizing an homology between $x_k \otimes z_r$ and $ x_h \otimes z_r$ differs from $\beta$ by the boundary of an element in $A_1 \otimes B_1$. In other words, such a 1-chain should be of the form
\[ \partial \left( \sum_{i,j} \epsilon_{ij} y_i \otimes w_j \right) + \sum_{\ell=a}^{b-1} y_\ell \otimes z_r \]
for some coefficients $\epsilon_{ij}\in \F$. 
Obviously, we have that 
\begin{equation}\label{crucial}
E_s \left( \partial \left( \sum_{i,j} \epsilon_{ij} y_i \otimes w_j \right) + \sum_{\ell=a}^{b-1} y_\ell \otimes z_r \right) \geq E_s\left( \sum_{\ell=a}^{b-1} y_\ell \otimes z_r  \right)
\end{equation} 
provided that none of the generators $y_a\otimes z_r, y_{a+1}\otimes z_r, \dots , y_{b-1}\otimes z_r $ appears as a component of
\[ \partial \left( \sum_{i,j} \epsilon_{ij} y_i \otimes w_j \right) =\sum_{i,j} \epsilon_{ij} \partial y_i \otimes w_j + \sum_{i,j} \epsilon_{ij} y_i \otimes \partial w_j \ .\]
On the other hand if so happens for some $y_i \otimes z_r$, after cancellation the summand 
\[\sum_{j} \epsilon_{ij} y_i \otimes \partial w_j= \sum_j \epsilon_{ij} y_i \otimes  z_{j} + \epsilon_{ij}  y_i \otimes z_{j+1} \]
has a component of the form $y_i \otimes z_\mu$ for some $\mu \not= k$. Thus, since
$E_s( y_i \otimes z_\mu) =E_s(y_i)+ E_s(z_\mu)> E_s(y_i)+ \Upsilon_s(B_*)
=E_s(y_i)+ E_s(z_r)=E_s(y_i \otimes z_r)$, also in this case the inequality in (\ref{crucial}) holds, and we are done.  
\end{proof}

\begin{proof}[Proof of Proposition \ref{applicationthin}] According to Feller and  Krcatovitch \cite{feller} the Ozsv\' ath-Stipsicz-Szab\' o upsilon function $\Upsilon_{p,q}(t)$ of the $(p,q)$ torus knot can be computed recursively by means of the formula
$ \Upsilon_{p,q}(t)=\Upsilon_{p-q, q}(t)+ \Upsilon_{q+1,q}(t)$.
Thus, $\Upsilon_K(t)=\Upsilon_{8,5}(t)-\Upsilon_{6,5}(t)- \Upsilon_{4,3}(t)=\Upsilon_{6,5}(t)+ \Upsilon_{4,3}(t)+ \Upsilon_{3,2}(t)-\Upsilon_{6,5}(t)- \Upsilon_{4,3}(t)=\Upsilon_{3,2}(t)$ proving that $K$ is an upsilon-alternating knot.  

Now suppose by contradiction that there exists a Floer thin knot $J$ such that $T_{6,5}\# T_{4,3} \sim T_{8,5} \# J$. The upsilon function of the torus knot $T_{6,5}$ has its singularities at $t=2/5,4/5,6/5, 8/5$ while the one of $J$ has its only singularity at $t=1$. The upsilon function of the torus knot $T_{4,3}$ and $T_{8,5}$ on the other hand both have a singularity at $t=2/3$. Thus, as consequence of Theorem \ref{sum} we have that
\[\Upsilon_{T_{4,3}, 2/3}^{(2)}\left(\frac{2}{3}\right)= 
\Upsilon_{T_{6,5}\# T_{4,3}, 2/3}^{(2)}\left(\frac{2}{3}\right)=
\Upsilon_{T_{8,5}\# J, 2/3}^{(2)}\left(\frac{2}{3}\right)=
\Upsilon_{T_{8,5}, 2/3}^{(2)}\left(\frac{2}{3}\right) \ .\]
We claim that $\Upsilon_{T_{4,3}, 2/3}^{(2)}\left(2/3\right)\not= 
\Upsilon_{T_{8,5}, 2/3}^{(2)}\left( 2/3\right)$. In fact, by Proposition \ref{secondaryupsilonstaircases} we have that 
$\Upsilon_{T_{4,3}, 2/3}^{(2)}\left(2/3\right) =-4/3$ and $\Upsilon_{T_{8,5}, 2/3}^{(2)}\left(2/3\right)=- 20/3$.
\end{proof}

\section{Algebraic Knots}\label{algebraic}
Suppose that $Z \subset \C^2$ is a planar complex curve given by the equation $f(x,y)=0$. Recall that a point $p \in Z$ is said to be \textbf{regular} if the partial derivatives 
$ \partial f/\partial x$  and $\partial f/\partial y$
do not both vanish at $p$. A point that is not regular is said to be \textbf{singular}. 
In what follows by an isolated plane curve singularity $(Z,p)$ we mean a planar complex curve $Z$ with an isolated singularity at $p \in Z$. Without loss of generality we can always suppose $p$ to be the origin of $\C^2$.

Let $(Z, 0)$ be an isolated  plane curve singularity. A small sphere $S^3_\epsilon(0)$ centred at the origin intersects $Z$ transversally in a link $K= S^3_\epsilon(0) \cap Z$. This is the \textbf{link} of the plane curve singularity and in a neighbourhood of the origin $Z$ looks like a cone over it. If the link $K \subset S^3$ is actually a knot we say that $(Z,0)$ is \textbf{cuspidal}. Knots arising from this construction are called \textbf{algebraic knots}. 

Naturally attached to a plane curve singularity $(Z,0)$ there is an arithmetic object capturing informations about the complex geometry of its germ.  
Given an analytic parametrization $\varphi(z)$ of $Z$ around $0$ consider the pull-back homomorphism $\varphi^*: \C[[x,y]] \to \C[[z]]$ defined by $g\mapsto g \circ \varphi$. Set
\[ S= \big\{s \in \Z_{\geq 0} \ | \ g(\varphi(z))= z^sh(z)\text{ for some } g\in \C[[x,y]], \ h \in \C[[z]]  \text{ with } h(0) \not= 0  \big\} \ .\] 
It is esay to see that $S$ is a \textbf{semigroup}, meaning that $0 \in S$ and  if $a, b \in S$ so is $a+b$. The semigroup of a cuspidal singularity $(Z,0)$ is related to its knot $K$  via the Alexander polynomial
\begin{equation} \label{relationwithsemigroup}
\Delta_K(t)= \sum_{s \in S } t^s- t^{s+1} 
\end{equation} 
Notice that this is a finite sum since the semigroup of a plane curve singularity eventually covers all the positive integers.
  
Any cuspidal plane curve singularity $(Z,0)$ has a parametrization of the form $x=z^a$, $y=z^{q_1}+ \dots + z^{q_n}$ for some positive integers $q_1<q_2< \dots <q_n$. Such a representation is unique if we further assume $gcd(a, q_1, \dots , q_i)$ to not divide $q_{i+1}$ and $gcd(a, q_1, \dots , q_n)=1$. The sequence $(a; q_1, \dots , q_n)$ is the \textbf{Puiseaux characteristic sequence} of the cuspidal singularity $(Z,0)$, and the number $a$ is its \textbf{Puiseaux exponent}. It is a fundemantal fact of the theory of plane curve singularities \cite[Chapter 5]{wall} that starting from the Puiseaux characteristic sequence of a cuspidal singularity one can reconstruct both its semigroup and the topology of its link. 

\begin{thm}\label{cuspidal}Let $(Z,0)$ be a cuspidal plane curve singularity with Puiseaux characteristic sequence $(a;q_1, \dots , q_n)$. Set $D_i=gcd(a,q_1, \dots , q_i)$, $s_1=q_1$ and 
\[ s_i= \frac{aq_1+D_1(q_2-q_1)+ \dots + D_{i-1}(q_i-q_{i-1}) }{D_{i-1}} \]
for $i=0, \dots ,  n-1$. Then the link $K$ of $(Z,0)$ is the $(n-1)$-fold iterated cable of the $(a/D_1,q_1/D_1)$ torus knot with cabling coefficients $(D_{i-1}/D_i, s_{i-1}/D_i)$, $i=2, \dots, n$. 
Furthermore, the semigroup  of $(Z,0)$ is generated by $\{a, s_1, \dots , s_n\}$. \QEDB
\end{thm}

From the viewpoint of Heegaard Floer theory, algebraic knots are interesting since they provide a good source of examples of $L$-space knots \cite{NemGor}. Because of Equation \ref{relationwithsemigroup}, the staircase of an algebraic knot can be recovered from the semigroup of its singularity. More precisely, suppose that $(Z,0)$ is a plane curve singularity giving rise to a genus $g$ algebraic knot $K$. The semigroup $S$ of $(Z,0)$ determines a colouring of $\{0, \dots , 2g-1\}$: color by red the numbers in $S \cap \{0, \dots , 2g-1\} $ and by blue the ones in its complement $(\Z \setminus S) \cap  \{0, \dots , 2g-1\}$. By counting the gaps between blue and red numbers as suggested by Figure \ref{semigroup} we get two sequences of numbers $r_1, \dots r_g$ and $b_1, \dots , b_g$. As a consequence of the general recipe discussed at the beginning of Section \ref{L-spaceknots} one can see that $CFK^\infty(K)\simeq  S_*(r_1,b_1, \dots , r_g, b_g)$.

\begin{figure}
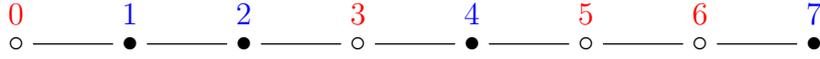

\hspace{0.1cm}
\xygraph{
!{<0cm,0cm>;<1cm,0cm>:<0cm,1cm>::}
!~-{@{-}@[|(2.5)]}
!{(-2.7,1.5) }*+{\bullet}="x1"
!{(-4.2,1.5) }*+{\circ}="x2"
!{(-5.7,1.5) }*+{\circ}="x3"
!{(-7.2,1.5) }*+{\bullet}="x4"
!{(-8.7,1.5) }*+{\circ}="x5"
!{(-10.2,1.5) }*+{\bullet}="x6"
!{(-11.7,1.5) }*+{\bullet}="x7"
!{(-13.2,1.5) }*+{\circ}="x8"
!{(-2.7,1.9) }*+{\textcolor{blue}{7}}
!{(-4.2,1.9) }*+{\textcolor{red}{6}}
!{(-5.7,1.9) }*+{\textcolor{red}{5}}
!{(-7.2,1.9) }*+{\textcolor{blue}{4}}
!{(-8.7,1.9) }*+{\textcolor{red}{3}}
!{(-10.2,1.9) }*+{\textcolor{blue}{2}}
!{(-11.7,1.9) }*+{\textcolor{blue}{1}}
!{(-13.2,1.9) }*+{\textcolor{red}{0}}
"x2"-"x1"
"x3"-"x2"
"x4"-"x3"
"x5"-"x4"
"x6"-"x5"
"x7"-"x6"
"x7"-"x8"
}
\caption{\label{semigroup} The semigroup of the plane curve singularity $x^5+y^3=0$ is generated by $5$ and $3$. Its link is the torus knot $T_{5,3}$. The associated staircase can be computed from the colouring above by counting the gaps between blue and red numbers. In this case  $r_1=1, r_2=1, r_3=2$, $b_1=2, b_2=1, b_3=1$, and $CFK^\infty(T_{5,3})= S_*(\textcolor{red}{1},\textcolor{blue}{2},\textcolor{red}{1},\textcolor{blue}{1}, \textcolor{red}{2}, \textcolor{blue}{1})$.}
\end{figure} 

\subsection{$L$-space pretzel knots}
We now proceed to the proof of Theorem \ref{applicaton}. Suppose that $C$ is a south-west region. For every $x \in \R$ we can consider the truncated south west region $C_x= C \cap \{ A \leq x\}$.  This leads to a one-parameter family of upsilon type invariants $\Upsilon^{C_x}(K_*)$. Since $C_x \subseteq C$ we have that $\Upsilon^{C_x}(K_*) \leq \Upsilon^C(K_*)$. Furthermore for $x$ large enough $\Upsilon^{C_x}(K_*)= \Upsilon^C(K_*)$. Set 
\[\eta_C(K_*)= \min \left\{ x \ | \  \Upsilon^{C_x}(K_*) = \Upsilon^C(K_*) \right\} \ .\] 
Obviously, this is an invariant of stable equivalence. For a knot $K \subset S^3$ denote by $\eta_C(K)= \eta_C(CFK^\infty(K))$ the associated knot concordance invariance.

\begin{lem} \label{cruciallemma} Suppose that $(Z,0)$ is a cupsidal plane curve singularity with Puiseaux sequence $(a;q_1, \dots , q_n)$. Denote by $K$ the algebraic knot associated to $(Z,0)$ and by $S$ its semigroup. Let $n(S)$ be the maximum among the integers $n \geq 0$ such that 
\[S \cap \Z_{\leq na} = \{0, a, 2a, \dots, na \} \ .\]
Choose $C= \{ 1/a \cdot  A +(1 - 1/a) \cdot j \leq 0 \}$. Then, 
\[ \eta_C(K)=\left(1 - \frac{1}{a} \right) \tau(K)- (a-1) \  n(S) \ . \]
\end{lem}
\begin{proof}
Let $g$ be the genus of $K$. Colour by red and blue the numbers in $\{0, 1, \dots, 2g-1 \}$ as specified by $(R=S \cap\{0, 1, \dots, 2g-1 \}, B= (\Z \setminus S) \cap\{0, 1, \dots, 2g-1 \})$ and by recording the gaps between red and blue numbers form the sequences $r_1, \dots , r_g$ and $b_1, \dots , b_g$ suggested by Figure \ref{semigroup}. 
By the definition of $n(S)$ we have that $r_1=\dots = r_n=1$, $b_1= \dots = b_n= a-1$, and $1\leq b_i<a-1$ for $i=n+1 ,  \dots , g$. Thus, $CFK^\infty(K) \simeq S_*(1, a-1, \dots , 1, a-1, r_{n+1}, b_{n+1}, \dots , r_g, b_g)$ where the pair $(1,a-1)$ repeats $n=n(S)$ times. Denote by $x_0, \dots, x_g$ the Maslov grading zero generators of the staircase for $CFK^\infty(K)$. Similarly, denote by $y_1, \dots , y_g$ its Maslov grading one generators. 

Consider the half-space $C_\gamma$ of the $(A,j)$ plane defined by
$1/a \cdot  A +(1 - 1/a) \cdot j \leq \gamma$. We claim that for $\gamma= \Upsilon^C(K)=- \/2 \cdot  \Upsilon_{2/a}(K)=\tau(K)/a$ the only  Maslov grading zero generators contained in $C_\gamma$ are $x_1, \dots , x_n$. For this purpose define 
\[E(x_i)=1/a \cdot  A(x_i) +(1 - 1/a) \cdot j(x_i) \ . \] 
Obviously $E(x_i) \leq \gamma$ if and only if $x_i$ is in $C_\gamma$. A quick computation reveals that $E(x_i)=\gamma$ for $i=1, \dots , n$ and consequently that $x_1, \dots , x_n$ are actually contained in the boundary line of $C_\gamma$. We claim that $E(x_i) > \gamma$ for any $i >n$. For $k \geq1$ we have that 
\begin{align*}
E(x_{n+k}) &= E(x_n)- \frac{1}{a} \sum_{i=1}^k b_{n+i} + \left( 1 - \frac{1}{a} \right) \sum_{i=1}^k r_{n+i} \\
&= \gamma + \frac{1}{a} \left( -\sum_{i=1}^k b_{n+i} + \left( a - 1 \right) \sum_{i=1}^k r_{n+i} \right) \ . 
\end{align*} 
On the other hand, 
\[ \left( a - 1 \right) \sum_{i=1}^k r_{n+i} \geq k \cdot (a-1) > \sum_{i=1}^k b_{n+i}  \ , \]
proving that $E(x_{n+k}) \geq \gamma + 1/a \cdot (\text{something positive})>\gamma$, and we are done.

Since the only generators with Maslov grading zero in $C_\gamma$ are $x_1, \dots , x_n$ we conclude that $C_\gamma \cap \{ A \leq x+ \gamma\}$ contains a cycle generating $H_0(CFK^\infty(K))$ provided $x + \gamma \geq \min \{A(x_1) , \dots , A(x_n) \}=g-n(a-1)$. Thus, $\eta_C(K)+\gamma=g-n(a-1)$. Plugging in $\gamma= 1/a \cdot \tau(K)$, $g= \tau(K)$, and $n=n(S)$ the claim follows. 
\end{proof}

\begin{proof}[Proof of Theorem \ref{applicaton}]
Using the skein relation at a negative crossing we find that the symmetrized Alexander polynomial $\Delta_q(t)$ of a $P(-2,3,q)$ pretzel knot ($q \geq 7$ odd) is given by $\Delta_q (t)= (t-1+t^{-1})\Delta_{2,q}(t) + (t^{\frac{1}{2}}-t^{-\frac{1}{2}})\Delta_{2,q+3}(t) $, where $\Delta_{2,p}(t)$ denotes the Alexander polynomial of the $(2,p)$ torus link. Since $P(-2,3,q)$ is an $L$-space knot, this leads  to the conclusion that \[CFK^\infty(P(-2,3,q))\simeq S_*(1,2,1, 1, \dots , 1, 1,2,1) \  , \]
from where one computes $\tau(P(-2,3, q))=(q+3)/2$, and $\eta_C(P(-2,3,q))= (q-3)/3$ for $C= \{ 1/3 \cdot  A +2/3 \cdot j \leq 0 \}$. Notice that $\Upsilon_{P(-2,3,q)}(t)=-2 \cdot \Upsilon_t(S_*(1,2,1,\dots,1,2,1))$ has its only singularities at $t=2/3, 1, 4/3$.
 
Suppose by contradiction that for some $q \geq 7$ odd the pretzel knot $P(-2, 3, q)$ is concordant to a sum of algebraic knots $K_1 \# \dots \# K_m$.  For $i=1, \dots, m$ let $(Z_i,0)$ be a  plane curve singularity with knot $K_i$. Denote by $S_i$ the semigroup of $(Z_i,0)$ and by $a_i$ its Puiseaux exponent. According to Wang \cite{wang1} the Ozsv\' ath-Stipsicz-Szab\' o upsilon invariant $\Upsilon_K(t)$ of an algebraic knot has its first singularity at $t=2/a$ where $a$ denotes its Puiseaux exponent. Since $ \Upsilon_{P(-2,3,q)}(t)= \Upsilon_{\#_i K_i}(t) = \sum_i \Upsilon_{K_i}(t)$, and $\Delta \Upsilon'_{K_i}(t) \geq 0$, this leads to the conclusion that either $a_i=3$, or $a_i=2$. 

Notice that, as consequence of Theorem \ref{cuspidal},  if $K=((T_{p,q})_{p_1, q_1} \dots )_{p_n, q_n}$ is the knot of a cuspidal plane curve singularity $(Z,0)$ with Puiseaux exponent $a$ then $a=p\cdot (p_1 \dots p_n)$. Since for every $i$ the Puiseaux exponent of $K_i$ is either 3 or 2, we conclude that $K_i$ is either a $(3,p)$ torus knot, or a $(2,k)$ torus knot.

An argument along the line of \cite[Theorem 6.2]{Livingston1} reveals that $\eta_C(K_1 \# \dots \# K_m)= \eta_C(K_1) + \dots +\eta_C(K_m)$. A direct computation shows that for the $(2,2k+1)$ torus knot $\eta_C=2/3 \cdot k$. Thus, as consequence of Lemma \ref{crucial} we have that 
\[ \eta_C(P(-2,3, q))= \eta_C(K_1 \# \dots \# K_m)= \frac{2}{3}\sum_{i} \tau(K_i)-2 \sum_j n(S_j)\ , \]
where the second sum is extended only to the $(3,p)$ torus knot summands. Plugging in   $\sum_{i} \tau(K_i)= \tau(K_1 \# \dots \# K_m)= \tau(P(-2,3, q))=(q+3)/2$ and $\eta_C(P(-2,3,q))= (q-3)/3$ we get that $-2=-2\sum_j n(S_j)$ and consequently that $K_1 \# \dots \# K_m$ is either of the form $T_{3,4}\# J$, or $T_{3,5} \#J$ where $J$ is a sum of $(2,n)$ torus knots and hence alternating. This leads to a contradiction since a knot of this form has $\tau= - \sigma/2$ while for a pretzel knot of the form $P(-2, 3, q)$ we have $(q+3)/2 =\tau \not= -\sigma/2=(q+1)/2$. 
\end{proof}

{\small 
\subsection*{Acknowledgements} The author would like to thanks Andr\' as Stipsicz, Andr\' as N\' emethi, and Paolo Aceto. The author was partially supported by the NKFIH grant K112735.
\par}

\bibliography{Bibliotesi}
\bibliographystyle{siam}

\end{document}